\documentclass[12pt]{amsart}
\usepackage[a4paper]{geometry}
\usepackage{xcolor}
\usepackage[all]{xy}
\usepackage{hyperref}
\usepackage{array}   
    \newcolumntype{L}{>{$}l<{$}} 
\usepackage{amssymb}

\newtheorem{theorem}{Theorem}[section]
\newtheorem{proposition}[theorem]{Proposition}

\newtheorem{conjecture}[theorem]{Conjecture}

\newtheorem{lemma}[theorem]{Lemma}
\theoremstyle{definition}

\newtheorem{problem}[theorem]{Problem}
\theoremstyle{remark}

\newcommand{\ZZ}{\mathbb{Z}}
\newcommand{\QQ}{\mathbb{Q}}

\newcommand{\OO}{\mathcal{O}}
\DeclareMathOperator{\disc}{\mathrm{disc}}

\DeclareMathOperator{\Res}{\mathrm{Res}}

\newcommand{\bfs}{\mathbf{s}}
\newcommand{\Prob}{\textnormal{Prob}}

\DeclareMathOperator{\sign}{sign}

\DeclareMathOperator{\Imp}{Imp}

\DeclareMathOperator{\rank}{rank}

\title{Probabilistic Galois Theory -- The Square Discriminant Case}
\author{Lior Bary-Soroker}
\address{School of Mathematical Sciences, Tel Aviv University, Tel Aviv 69978, Israel}
\email{barylior@tauex.tau.ac.il}
\author{Or Ben-Porath}
\address{School of Mathematical Sciences, Tel Aviv University, Tel Aviv 69978, Israel}
\email{orbenporath@mail.tau.ac.il}

\author{Vlad Matei}
\address{School of Mathematical Sciences, Tel Aviv University, Tel Aviv 69978, Israel  \bigskip \indent {Institute of Mathematics of the Romanian Academy\\  Bucharest\\Romania}}
\email{vladmatei@mail.tau.ac.il}

\date{\today}

\begin{document}

\begin{abstract}
    The paper studies the probability for a Galois group of a random polynomial to be $A_n$. We focus on the so-called large box model, where we choose the coefficients of the polynomial independently and uniformly from $\{-L,\ldots, L\}$. 
    The state-of-the-art upper bound is $O(L^{-1})$, due to Bhargava. We conjecture a much stronger upper bound  $L^{-n/2 +\epsilon}$, and that this bound is essentially sharp. 
    We prove strong lower bounds both on this probability and on the related probability of the discriminant being a square. 
\end{abstract}

\maketitle

\section{Introduction}
In its most naive sense, Probabilistic Galois Theory is an area of mathematics that studies the arithmetic properties of a random polynomial $f$ with integral coefficients.  One basic arithmetic property is the irreducibility of $f$ in the ring $\QQ[X]$. More generally, every polynomial $f$ comes with a Galois group, which we denote by $G_f$. It acts naturally on the roots of $f$, hence we may view it as a subgroup of $S_{\deg f}$, which is well-defined when $f$ is separable (up to conjugation). The Galois group encodes the arithmetic properties of $f$, e.g., $f$ is irreducible if and only if $G_f$ is transitive. 

In this paper, we study the so-called \emph{large box model}, in which we fix the degree $n=\deg f$ of $f$ and we choose  the  coefficients   uniformly and independently from $[-L,L]\cap \ZZ$, and we are interested in the probabilities as $L$ goes to $\infty$. See \cite{bary2020irreducibility,bary2020irreducible,breuillard2019irreducibility} for recent progress on the \emph{restricted coefficient model} in  which  $L$ is fixed and $n\to \infty$ and \cite{eberhard2022characteristic,ferber2022random} for recent progress in the \emph{random matrix model} (where $f$ is the characteristic polynomial of a random matrix).

The general principle is that Galois groups tend to be the largest possible. Van-der Waerden \cite{van1936seltenheit} proved that $S_n$ is the most probable Galois group:
\[
    \lim_{L\to \infty} \Prob(G_f = S_n) =1.    
\]
In the same paper van-der Waerden conjectured that the second most probable group is $S_{n-1}$: 
\[
     \Prob(G_f = S_{n-1} ) \sim \Prob(G_f\neq S_n), \quad L\to \infty.
\]
Since $G_f \leq  S_{n-1}$ if and only if $f$ has a root $\alpha$ in $\QQ$, one may, using the van der Waerden theorem for $f(X)/(X-\alpha)$,  compute the LHS:
\[
    \Prob(G_f = S_{n-1} ) \sim \Prob(f \mbox{ has a root in } \mathbb{Q}) \sim \frac{c_n}{L} , \quad L\to \infty, \quad n>2
\]
for an explicit constant $c_n>0$, see
\cite{chela1963reducible,kuba2009distribution}. As stated, the conjecture is only known for $n=1,2$ (trivially) and $n=3,4$  \cite{Chow2020EnumerativeQuartics}. For general $n$, there is a line of results towards the conjecture in terms of bounding $\Prob(G_f\neq S_n)$: 
In 1936, van der Waerden \cite{van1936seltenheit} gave the aforementioned result. In 1956, Knobloch \cite{Knobloch} obtained the first power saving. In 1973, Gallagher \cite{gallagher1973large} used the large sieve to show that $\Prob(G_f\neq S_n)=O(L^{-1/2+\epsilon})$. In 2010, Zywina \cite{zywina2010hilbert} applied  the larger sieve to remove the $\epsilon$. In 2013, Dietmann \cite{dietmann2013probabilistic} introduced a different approach, reducing the problem to counting integral points on varieties, and improved the bound to $O(L^{\sqrt{2}-2+\epsilon})$. In 2021, Bhargava \cite{bhargava2021galois} made a breakthrough by establishing  a weak version of the van der Waerden conjecture: $\Prob(G_f\neq S_n) = O(L^{-1})$. See also the two preceding results \cite{anderson2021quantitative, chow2021towards} in the  same year. 
In all of the above, the most challenging part is to bound the probability that $G_f=A_n$. More precisely, Bhargava proved 
\begin{align}
    \nonumber &\Prob(G_f \neq  S_{n}) = \Prob(G_f=S_{n-1}) + \Prob(G_f=A_n) +O(L^{-2}), \qquad \mbox{and} \\
    &\Prob(G_f=A_{n})=O(L^{-1}), \label{Bhar_A_n}
\end{align}
for $n\geq 10$ (for $n<10$, one needs to modify the error term in the first equation).

Even though the bound \eqref{Bhar_A_n} is a huge breakthrough, we do not believe that it is sharp. For small $n$'s the situtation is rather clear: 
If $n=1$, then $\Prob(G_f = A_1)=1$. If $n=2$, then $\Prob(G_f = A_2)=\Prob(f \mbox{ reducible})$ which equals $L^{-1}\log L$ up to a constant. The situation is less trivial when $n=3$, in which case, Xiao \cite{xiao2022monic} proved that $\Prob(G_f = A_3) \ll L (\log L)^2$. 
The goal of this paper is to suggest the following conjecture on the asymptotics of $\Prob(G_f=A_{n})$ for large $n$: 
\begin{conjecture}\label{conj}
    Let $n\geq 4$. Then
    \[
        L^{-n/2+\epsilon} \gg \Prob(G_f = A_n ) \gg L^{-n/2},
    \]
    as $L\to \infty$. 
\end{conjecture}

The remainder of the introduction is devoted to a discussion of the evidence for this conjecture. In summary, we establish a new lower bound for $\Prob(G_f=A_n)$ (Theorem~\ref{thm:An}). For $\Prob(\disc f = \square\neq 0)$ we obtain a much sharper lower bound (Theorem~\ref{thm:C_2WrS_n})). Heuristically, the latter suggests a sharp lower bound for $\Prob(G_f=A_n)$, see Problem~\ref{prob:1}.

For a separable polynomial $f$, we have $G_f\leq A_n$ if and only if $\disc f=\square$; so $\Prob(G_f = A_n)\leq \Prob(\disc f = \square\neq 0)$. 

A very naive heuristic argument suggests the size of the latter: By \cite{mahler1963two}, $\disc f$ is typically of size $L^{2n-2}$, so if we assume that $\disc f$ is a random number of this size,  $\Prob(\disc f = \square\neq 0) \leq  L^{-(n-1)}$.
This heuristic is too naive and gives a wrong answer. The reason is that there are many symmetries. We give two families of polynomials with square discriminant featuring  different types of symmetries. 

The first type of symmetries, stems from the following simple example: if $f+f'=g^2$ and $n\equiv 0\mod 4$, then
\[
    \disc(f)= \Res(f,f')=\Res(f,f+f')=\Res(f,g^2)=\Res(f,g)^2.
\]
It turns out that almost surely, in a precise quantitative sense, $G_f=A_n$. This idea traces back at least to Hilbert \cite{Hilbert}, who constructed polynomials with $G_f=A_n$ by making $f'$ close to a square, cf.\ \cite{LLT}. We make this count precise and prove:

\begin{theorem}\label{thm:An}
    Let $n\geq 4$ and put $k=\frac{1}{2}\left\lfloor \frac{n+1}{2}\right\rfloor$. Then 
    $\Prob(G_f = A_n) \gg L^{k-n} $.
\end{theorem}

The second type of  symmetries comes from decomposable polynomials. The simplest family, which already gives the best lower bounds to date up to a constant, is the family of polynomials  
$f(X)=g(X^2)$, with $g(0)\neq 0$ and $g$ separable.  The discriminant is a square if and only if $g(0)= \square$. This gives a linear condition in terms of the coefficients, but we have half as many coefficients. This gives the lower bound $\Prob(\disc(f)=\square) \gg L^{-\frac{n+1 }{2}}$.
More generally, one may take  compositions with a quadratic rational function, this improves only the implied constant, and not the order of magnitude. 

We stress that in this family, $G_f\neq A_{n}$. In fact, almost always $G_f$ is the Coxeter group of type $D_{n/2}$. Indeed, the roots of $f$ come in pairs $\{\lambda,-\lambda\}$, and the Galois action acts on the set of pairs. So $G_f= (C_2\wr S_{n/2})\cap A_n\leq C_2\wr S_{n/2}$, where $C_2\wr S_{n/2}$ is the Coxeter group of type $B_{n/2}$, also known as, the permutational wreath product endowed with the imprimitive action (for details, see Section~\ref{sec:lowerbound}).  

The imprimitivity allows one to give strong upper bounds on polynomials with this Galois group: Widmer \cite{widmer2011number} proves an upper bound  of the form $\ll L^{-\frac{n}{2}+1}$ (in fact, Widmer considers a non-monic model, but his arguments carry-over to the monic case). We improve the upper bound of Widmer by a factor of $L^{-\frac2n}$. In summary, we prove the following:

\begin{theorem}\label{thm:C_2WrS_n}
    Let $n\geq 6$ be even. Then
    \[
        L^{-\frac{n}{2}-\frac{2}{n}+1+\epsilon }\gg\Prob(G_f = (C_2\wr S_{n/2})\cap A_n) \gg L^{-\frac{n+1 }{2}}.
    \]
    In particular, $\Prob(\disc f=\square)\gg L^{-\frac{n+1}{2}}$.
\end{theorem}

When $n=4$ one has $(C_2\wr S_{4/2})\cap A_4=V_4$, and  Chow and Dietmann \cite{Chow2020EnumerativeQuartics} obtained the better upper bound $L^{-2+\epsilon} \gg \Prob(G_f = V_4)$. Their lower bound is similar to ours. 

From the naive general principle that Galois groups should be as large as possible, one may believe that the answer to the following problem is positive:
\begin{problem}\label{prob:1}
    As $L\to \infty$, is $\Prob(\disc f =\square) \sim \Prob( G_f = A_n)$?
\end{problem}

So Theorem~\ref{thm:C_2WrS_n} may be considered as an evidence toward the lower bound of Conjecture~\ref{conj} for even $n$-s. 
We do not have good evidence for the upper bound.

We formulate a speculative open problem:
\begin{problem}
    Let $G_1\leq G_2\leq S_n$. Then, there exists a constant $C>0$ such that 
    \[
        \Prob(G_f=G_1)\leq C\cdot  \Prob(G_f = G_2),
    \]
    for all $L\geq 1$.
\end{problem}
A negative answer to this problem  will be a big surprise. On the other hand, an affirmative answer in full generality will imply the inverse Galois problem (indeed take $G_1=1$).

As discussed above, the upper bound in Theorem~\ref{thm:C_2WrS_n} relates to Widmer's work on imprimitive Galois groups. In fact, Theorem~\ref{thm:C_2WrS_n} follows from an improvement of Widmer's general result:  Assume  $n=n_1 n_2$, $n_1,n_2>1$ and let $G$ be a subgroup of $S_n$. We say that $G$ is \emph{transitive $(n_1,n_2)$-imprimitive} if $G$ is transitive and there is a partition  of $n$ into $n_1$ blocks $\Lambda_i$ such that for all $g\in G$ and $1\leq i\leq n$ there exists $1\leq j\leq n$ such that  $g\Lambda _i = \Lambda_j$. In particular, $|\Lambda_i|=n_2$ for all $i$. For example, the wreath product $C_2\wr S_{n/2}$ with the imprimitive action is a transitive $(2,n/2)$-imprimitive. We prove
\begin{theorem}\label{thm:imprimitive}
    Let $n=n_1 n_2$ with $n_1,n_2>1$. Then
    \[
        \Prob(G_f \mbox{ is transitive $(n_1,n_2)$-imprimitive}) \ll L^{-n+m+\epsilon},
    \]
    where $m = \frac{n}{2}-\frac{2}{n}+1$ if $n_2=2$ and $n\geq 10$, and
    \[
        m = \max\bigg\{n_2+\frac{1}{2}+\frac{1}{n_1},n_1+(1-\frac{1}{n_1})(n_2-1),n_1+\frac{n_2}{2}\bigg\},
    \]
    otherwise.
\end{theorem}
The upper bound in Theorem~\ref{thm:C_2WrS_n}  follows immediately from Theorem~\ref{thm:imprimitive} when $n\geq 10$. For general $n\geq 4$, both follow from a  technical lemma (Lemma~\ref{lem:mathcalF_bounds}).  

Finally, we compare our result with Widmer's. One minor difference is that we consider monic polynomials and Widmer non-monic. More importantly, Widmer's bound is $L^{-n+n_1+n_2-1}$ while ours is $L^{-n+m+\epsilon}$. If $n_1,n_2>2$ or if $n_2=2$ and $n\geq 10$, then $m<n_1+n_2-1$, hence our bound is sharper. In the other case, when $n_1=2$ or $n<10$ and $n_2=2$, then $m=n_1+n_2-1$, and the results coincide.

If $G_f$ is transitive $(n_1,n_2)$-imprimitive, and $\alpha$ is a root of $f$, then by elementary group theory and the fundamental theorem of Galois theory there exists a subfield $\QQ\subset E \subset \QQ(\alpha)$ of degree $n_1$ over $\QQ$. Widmer exploits this fact to split the polynomial count according to $E$, see \cite[Eqs.~6.8 and 6.9]{widmer2011number}. We refine Widmer's count by grouping number fields according to the discriminant; subfields with large discriminant contribute less polynomials, and subfields with small discriminant are less numerous.

\subsection*{Outline of the paper}
In Section~\ref{sec:proof12} we prove Theorem~\ref{thm:An}. 
Section~\ref{sec:lowerbound} is devoted to prove the lower bound in Theorem~\ref{thm:C_2WrS_n}. Finally, the upper bound in Theorem~1.3 and Theorem~\ref{thm:imprimitive},  both of which  necessitate the development of new tools, are proven in Section~\ref{sec:UpperBounds}.

\section*{Acknowledgments}
The authors thank J.P. Serre for useful comments, in particular, for suggesting that Proposition~\ref{prop:wreath_product:alternating_geometric_realization} follows from the basic theory of  invariants of Coxeter groups.

The authors were partially supported by the Israel Science Foundation (LBS and OBP by grant no.~702/19 and VM by grant no.~ 2507/19).
VM was also partially supported by the CNCS-UEFISCDI grant PN-III-P4-ID-PCE-2020-2498.

\section{Proof of Theorem~\ref{thm:An}}\label{sec:proof12}
Let $K$ be a field of characteristic $0$, $n>1$ an integer, and $r=\left\lfloor \frac{n-1}{2}\right\rfloor$.
Let $a_1,\ldots, a_{r}$ and $t$ be independent variables and put
\[
    h = X^{r}+ a_1 X^{r-1}+\ldots+ a_r.
\]

We define  two families of polynomials $f=f(a_1,\ldots,a_r,t)$ with Galois group $A_n$, one for even degrees and one for odd degrees. For $n$ even, we define $f$ by the equations
\[
    f'=nXh^2, \qquad (-1)^{n/2}f(0)=t^2.
\]
Note that $f'(0)=0$ in this case. For $n$ odd, we define $f$ by
\[
    f-X f' = -(n-1)(X-1) h^2, \qquad (-1)^{(n-1)/2} f'(1) = t^2.
\]
In both cases,  $f\in\QQ[X]$ is uniquely determined by $a_1,\ldots, a_r$ and $t$.   

When $n$ is even, \cite[Lemma 3.1]{LLT} implies that the discriminant is a square. (Our $f$ coincide with $\tilde f_{\gamma}$ in \emph{loc.\ cit.\ }notation, for $\gamma = (-1)^{n/2}t^2$. Also, note that the discriminant of $\tilde f_{\gamma}$ is a square if and only if the discriminant of  $f_{\gamma}$ is).
Thus, the Galois group $G_{f,{\rm gen}}$ of $f$ over $\mathbb{Q}(a_1,\ldots, a_r,t)$ is contained in $A_n$. On the other hand, by \cite[Lemma 3.2]{LLT}, there exists a specialization of the coefficients of $f$ to $\mathbb{Q}$ such that the specialized  polynomial (denoted by $\tilde{P}_{\gamma}$ in \emph{loc.\ cit.}) has Galois group $A_n$. Thus, $G_{f,{\rm gen}}$ has $A_n$ as a subquotient, i.e. $|G_{f,{\rm gen}}|\geq |A_n|$. Thus, $G_{f,{\rm gen}}=A_n$. 

When $n$ is odd, we apply similar arguments with \cite[Lemma 3.6]{LLT} instead of \cite[Lemma 3.1]{LLT} and \cite[Lemma 3.7]{LLT} instead of \cite[Lemma 3.2]{LLT}. The conclusion is the same, that is,  $G_{f,{\rm gen}}=A_n$. (Our $f$ is the same as $\tilde{f}_{(-1)^rt^2}$ in \emph{loc.\ cit.} and $\alpha=1$ in our case.)

Choose values $\bfs:= (\bar{a}_1,\ldots, \bar{a}_r,\bar{t})\in \ZZ^{r+1}$ for the coefficients, and denote by $f_{\bfs}, h_{\bfs}$ the corresponding polynomials. Note that $f_{\bfs},h_{\bfs}$ have integer coefficients whenever $n!\mid a_1,\ldots,a_r$. Hence a positive proportion of the $f_{\bfs}$ are integral. By Hilbert's irreducibility theorem, for  most choices of $\bfs$, the Galois group of $f_{\bfs}$ remains $A_n$. To be more precise, put 
\[
    \Omega=\Omega(n,L)=\left\{\bfs\in \mathbb{Z}^{r+1}:|\bar{a}_i|, |\bar{t}|\leq L^{1/2} , G_{f_{\bfs}}=A_n \right\}. 
\]
Then, by \cite{Cohen}, we have 
\[
    |\Omega| = L^{\frac{r+1}{2}}\Big(1+O(L^{-1/2+\epsilon})\Big),
\]
as $L\to \infty$. 
For a polynomial $g = \sum g_iX^i$ with integer coefficients, let $H(g) = \max_i \{|g_i|\}$ be the height function. By the defining equation of $f$, we have that $H(f_{\bfs})\ll H(h_{\bfs}^2)+|t|^2$. Obviously,  $H(h^2)\ll H(h)^2\ll L$. Thus, if $\bfs\in \Omega$, then $H(h_{\bfs})\leq L^{1/2}$ and so  $H(f_{\bfs})\ll L$. Thus, the number of polynomials in $\{f_s:s\in\Omega\}$ of height $\leq L$ with Galois group $A_n$ is $\gg L^{\frac{r+1}{2}}$. Since a positive proportion of those are integral, the result follows.
\qed

\section{The lower bound in Theorem~\ref{thm:C_2WrS_n}}\label{sec:lowerbound}
We construct a generic family with Galois group $(C_2\wr S_{m})\cap A_n$, $n=2m$ and then apply Hilbert's irreducibility theorem to deduce the lower bound in Theorem~\ref{thm:C_2WrS_n}.

For the convenience   of the reader and for the sake of fixing notation as in \cite[\S4]{bary2020chebotarev}, we recall the construction of \emph{the permutational wreath product} $C_2 \wr S_m$ (also known as, the signed permutations group, the Weyl group of type B, or the hyperoctahedral group). The elements in the group are formal products $\xi \sigma$, where $\xi\colon \{1,\ldots, m\} \to \{\pm1\}$ and $\sigma \in S_m$. The product is induced by $(\xi_1 \xi_2) (i)=\xi_1(i) \xi_2(i)$ and $\sigma \xi = \xi^{\sigma^{-1}} \sigma$, where $\xi^{\sigma^{-1}}(i) := \xi(\sigma^{-1}i)$. The group comes with \emph{the imprimitive action} on pairs $(\epsilon , i)$, $\epsilon \in \{\pm 1\}$ and $1\leq i\leq m$, given by 
\begin{equation}\label{wreathproductaction}
    \xi\sigma . (\epsilon,i):=(\xi(\sigma i) \epsilon, \sigma i).
\end{equation}
Since $n=2m$, this induces an embedding $C_2 \wr S_m \to S_n$. Under this embedding, one can compute that $\sign (\xi\sigma)= \prod_{i=1}^m \xi(i)$. In other words, $(C_2\wr S_m)\cap A_n = \{\xi\sigma : \prod_i \xi(i)=1\}$.

\begin{proposition}\label{prop:wreath_product:alternating_geometric_realization}
    Let $n=2m>1$ be an even integer, let $K$ be a field of characteristic $0$. Let $B_0,\ldots, B_{m-1}$ be independent variables, and put 
    \[
        f_1(B_0,\ldots, B_{m-1},X) = X^m + B_{m-1} X^{m-1} + \cdots + B_1 X + (-1)^m B_0^2.
    \]
    Then, the Galois group of $f = f_1(X^2)$ over $K(B_0,\ldots, B_{m-1})$ is $(C_2 \wr S_m)\cap A_n$. 
\end{proposition}

As we learned from J.P.~Serre, since $(C_2 \wr S_m)\cap A_n$ is a Coxeter group of type $D_{m}$, this proposition follows immediately from the basic invariants of Coxeter groups. (To see this, note that the coefficients are homogeneous and invariant and apply \cite[Proposition on page 67]{humphreys1990reflection}.) Moreover, in this specific case, the proof is also an  elementary exercise in  Galois theory.  We include the details, for the convenience of the reader. 

\begin{proof}
    Let $\mathcal{P}_m\cong\mathbb{A}^m$ be the space of monic polynomials of degree $m$ and let $s\colon \mathbb{A}^m\to \mathcal{P}_m$ be the map given by $(y_1,\ldots, y_m)\mapsto \prod(X-y_i)$. 
    Let $\pi^m\colon \mathbb{A}^m\to \mathbb{A}^m$ be given by $(t_1,\ldots, t_m)\mapsto (t_1^2,\ldots, t_m^2)$. If $(C_0,\ldots, C_{m-1})$ are the coordinates of $\mathcal{P}_m$, that is, the elements of $\mathcal{P}_m$ are of the form $g_1(C_0,\ldots, C_{m-1},X) =X^m + C_{m-1}X^{m-1} + \cdots +C_0$, then $s\circ \pi$ is defined by the equation 
    \[
         g(C_0,\ldots, C_{m-1}, X)= 0,
    \]
    where $g(C_0,\ldots, C_{m-1},X) = g_1(C_0,\ldots, C_{m-1},X^2)$.
    By the construction in \S6.1 of \cite{bary2020chebotarev}, the Galois group of $g$ over $K(C_0,\ldots, C_{m-1})$ (equivalently, of the cover  $s\circ \pi$) is isomorphic to the imprimitive wreath product $C_2\wr S_m$. The roots of $g$ are $\pm \sqrt{Y_i}$, where $Y_i$ are the roots of $g_1$.  The construction comes with  an explicit action (Equation~23 in \emph{loc.\ cit.}):
    \begin{equation}\label{action}
        (\xi,\sigma). \sqrt{Y_i} := \xi(\sigma i) \sqrt{Y_{\sigma i}}.
    \end{equation}
    This action is compatible with the imprimitive action of the wreath product on pairs \eqref{wreathproductaction} under the identification $\pm \sqrt{Y_i} \leftrightarrow (\pm1 ,i)$. 
    Let $E$ be the fixed field of $(C_2\wr S_m)\cap A_n$ in the splitting field $K(\sqrt{Y_1},\ldots, \sqrt{Y_m})$ of $g$. Then, $E$ is a quadratic extension of $K(C_0,\ldots, C_{m-1})$ which we can compute readily: We have
    \[
        (\xi,\sigma).\sqrt{(-1)^m C_0} = (\xi,\sigma).\Big(\prod_{i=1}^m \sqrt{Y_i}\Big) = \prod_{i=1}^m \xi(\sigma i)\sqrt {Y_{\sigma i}}  = \sign(\xi,\sigma) \sqrt{(-1)^mC_0}.
    \]
    So $\xi\sigma\in A_n$ if and only if $\xi\sigma$ fixes $\sqrt{(-1)^m C_0}$. Thus, $E = K(\sqrt{C_0},C_1,\ldots, C_{m-1})$. Since the  transcendence degree of $E$ is $m$ (as an algebraic extension of the field of rational functions $K(C_0,\ldots, C_{m-1})$), we deduce that $B_0:=\sqrt{C_0}$, $B_1:=C_1$, $\ldots,$ $B_{m-1}:=C_{m-1}$ are independent variables. Moreover, $K(\sqrt{Y_1},\ldots, \sqrt{Y_m})$ is the splitting field of $f$ over $E=K(B_0,\ldots, B_{m-1})$. By the Galois correspondence, the Galois group of $f$ is $(C_2\wr S_m)\cap A_n$, as needed. 
\end{proof}

\begin{proof}[Proof of lower bound in Theorem \ref{thm:C_2WrS_n}]
    In the notation of Proposition~\ref{prop:wreath_product:alternating_geometric_realization}, let $N$ be the number of $(b_0,\ldots, b_{m-1})$ in the box
    \[
        \{b_i\in\ZZ: |b_i|\leq L, |b_0|\leq \sqrt{L}\}
    \]
    such that the Galois group of $f(b_0,\ldots, b_{m-1}, X)$ over $\mathbb{Q}$ is $(C_2\wr S_m)\cap A_n$. 
    
    By Hilbert's irreducibility theorem \cite[Theorem~A.2]{LLT}, we have that  $N\gg L^{m-1/2}$. Moreover $H(f)\leq L$. This concludes the proof for the lower bound. 
\end{proof}

\section{Upper Bounds}\label{sec:UpperBounds}

\subsection{Combinatorial lemma}
A \emph{cover} $T=\{ S_i\}$ of a set $S$ is a family of subsets with $S=\bigcup_i S_i$. It is called a \emph{partition} if the subsets are disjoint. For a cover $T$ we define 
\[
    M_n(T) := \sum_{S_i\in T} |S_i|^n.
\]
\begin{lemma}\label{lem_disjoint_covers_bound}
    Let $1\leq Y\leq X$, $n>1$, and $S$  a finite set of size $X$. Let $\mathcal{T}$ be the set of partitions $T= \{S_i\}$  of $S$ such that $|S_i|\leq Y$. Then,
    \[
        \max_\mathcal{T} M_n(T) =  X Y^{n-1} + O(Y^n).
    \]
    Moreover, the maximum is obtained when at most one nonempty set is of size $<Y$.
\end{lemma}

\begin{proof}
    Consider a partition $T=\{S_1,\ldots, S_k\}\in \mathcal{T}$ such that all $|S_1|=\cdots |S_{k-1}|=Y$ and $|S_k| = X-(k-1)Y \leq Y$. Since $k-1=\lfloor \frac{X-1}{Y}\rfloor$ we have,
    \[
        M_n(T) = \sum_{S_i\in T} |S_i|^n = \lfloor \frac{X-1}{Y} \rfloor Y^n + |S_k|^n \asymp XY^{n-1} + O(Y^n)
    \]
    We claim that these partitions maximize $M_n(T)$. Indeed, if $T$ is a partition not as above, then without loss of generality $|S_1|\leq |S_2| <Y$. Let $s\in S_1$, and consider the partition with $S_1' =S_1\smallsetminus \{s\}$, $S_2'=S_2\cup\{s\}$, and $S_i'=S_i$ for $i>2$. Then,
    \[
        M_n(T') = M_n(T) + |S_1'|^n + |S_2'|^n - |S_1|^n + |S_2|^n) >M_n(T),
    \]
    (Here we used that $(x+1)^n-x^n\geq nx^{n-1}\geq ny^{n-1} \geq y^n-(y-1)^n$, for all $x\geq y\geq 1$.) Hence, $M_n(T)$ is not maximal. 
\end{proof}

For our purposes we will need the bound for $M_n(T)$ for covers with small intersections, and not only partitions:

\begin{lemma}\label{lem_covers_bound}
    Let $1\leq Z\leq  Y\leq X$ and $R=o(Z)$ be positive real numbers, let $n>1$, let $S$ be a finite set of size $X$ and  let $\mathcal{T}$ be the set of covers $\bigcup_i S_i = S$ of $S$ such that for all $i$
    \[
        Z\leq |S_i|\leq Y \quad \mbox{and} \quad |S_i\cap(\cup_{j\neq i} S_j)| \leq R.
    \]
    Then,
    \[
        \max_\mathcal{T} M_n(T)=X Y^{n-1} (1+o(1)), \qquad Y\to\infty.
    \]
\end{lemma}

\begin{proof}
    For $T\in \mathcal{T}$, we attach a partition $T'$ such that $S_i' = S_i\smallsetminus \bigcup_{j<i} S_j$. Then, 
    \[
        M_n(T') \geq \sum_{S_i\in T} (|S_i|-R)^n = M_n(T) + O\left( |T| R Y^{n-1}\right), 
    \]
    where we used the inequality $(x-y)^n =x^n + O(x^{n-1} y)$, for $x\geq y$.
    Since $R=o(Z)$ and 
    $|T|\leq \frac{X}{Z-R} = O(X/Z)$, by Lemma~\ref{lem_disjoint_covers_bound}, we get that
    \[
        M_n(T)\leq M_n(T') + o(XY^{n-1})= XY^{n-1} (1+o(1)),
    \]
    as needed.
\end{proof}

\subsection{Auxiliary results on algebraic integers of bounded height}

\begin{lemma}\label{lem_gen}
    Let $n_1,n_2>1$, $K$ be a field, $b_0,\ldots,b_{n_2-1}$ algebraic elements over $K$ that generate a separable field extension $E/ K$ of degree $n_1$, and put $g= X^{n_2}+\sum_{i=0}^{n_2-1} b_iX^i$. Then there exists a constant $C=C(n_1,n_2)$ such that for any  $S\subseteq K$ of size $|S|\geq C$ there exists $s\in S$ such that $b=g(s)$ generates $E/K$. 
\end{lemma}

\begin{proof}  
    Since subfields of $E/K$ corresponds to subgroups of the Galois group of the Galois closure of $E/K$, there exist only $C_1=C_1(n_1)$ many of them.  
    
    Let $C = n_2C_1+1$ and $S\subseteq K$ of size $|S|\geq C$.
    Either $g(s)$ generates $E/K$ for some $s\in S$, or by the pigeonhole principle, there exist  $n_2+1$ many $s\in S$ for which $g(s)$ lie in the same proper subfield $F$ of $E$.
    Then by Lagrange interpolation, $b_i\in F$, and in particular do not generate $E$, contradiction. 
\end{proof}

Let $E$ be a number field and $b\in E$. Its absolute multiplicative Weil height may be given by 
\[
     H(b) := \prod_{\sigma\in {\rm Emb}(E, \mathbb{C})}\max(1,|\sigma b|)^{\frac{1}{[E:\QQ]}},
\]
see \cite[Proposition 1.6.6]{bombieri2007heights}.
Note that the product on the right-hand side is independent of the choice of $E$. There is a connection between the absolute multiplicative Weil height  of an algebraic integer $b$ and the height of its minimal polynomial $g$. Since each coefficient of $g$ is a sum of products of conjugates of $b$, where the number of summands and the number of terms in each summand are bounded in terms of $n$, we get that 
\begin{equation}\label{naiveHvsWh}
    H(g) \ll \prod_{\sigma\in {\rm Emb}(
    \QQ(b), \mathbb{C})} \max(1,|\sigma b|) = H(b)^{\deg g}.
\end{equation}

For $\sigma\in {\rm Emb}(E, \mathbb{C})$, we let $\sigma g$ be the polynomial obtained by applying $\sigma$ to the coefficients of $g$. 

\begin{lemma}\label{lem_coefficient_height}
    Let $b_0,\ldots,b_{n_2-1}$ be algebraic numbers that generate a number field $E$ of degree $n_1$. Put $g= X^{n_2}+\sum_{i=0}^{n_2-1} b_iX^i$ and $f=\prod_{\sigma \in {\rm Emb}(E,\mathbb{C})} \sigma g$. Then, for every $i\in \{0,\ldots, n_2-1\}$,
    \[
        H(b_i)\ll_{n_1,n_2} H(f)^{\frac{1}{n_1}}.
    \]
\end{lemma}

\begin{proof}
    Let $i\in \{0,\ldots, n_2-1\}$. 
    Since $g$ is monic, for every $ \sigma\in {\rm Emb}(E, \mathbb{C})$, we have  
    \[
        H(\sigma g) \geq \max(1,|\sigma b_i|).
    \]
    Since $H(f_1f_2)\gg H(f_1)H(f_2)$ for any $f_i\in \mathbb{C}[X]$, where the implied constant depends only on the degrees \cite[Eq.~II]{mahler1962some},  we get 
    \[
        H(f) 
        \gg \prod_{\sigma\in {\rm Emb}(E, \mathbb{C})} H (\sigma g)
        \geq \prod_{\sigma\in {\rm Emb}(E, \mathbb{C})} \max(1,|\sigma b_i|) = H(b_i)^{n_1},
    \]
    as needed.
\end{proof}

\begin{lemma}\label{lem_generatorsmallheight}
    Let  $b_0,\ldots,b_{n_2-1}$ be algebraic numbers (integers) that generate an extension $E$ of degree $n_1$. Then, there exists (an algebraic integer) $b\in E$ that generates $E/\QQ$ such that
    \[
        H(b)\ll_{n_1,n_2}\max(H(b_0),\ldots, H(b_{n_2-1})).
    \]
\end{lemma}

\begin{proof}
    We abbreviate $\ll_{n_1,n_2}$ and write $\ll$. Let $B=\max(H(b_0),\ldots, H(b_{n_2-1}))$. Put 
    \[
        f=\prod_{\sigma\in {\rm Emb}(E, \mathbb{C})}\sigma g
    \]
    where $g(X)=X^{n_2}+b_{n_2-1}X^{n_2-1}+\ldots+b_1 X+b_0$.
    
    For any polynomial $h$, define $h_{a}(X)=h(X+a)$ where $a\in \mathbb{Z}$. If $h\in \QQ[X]$, then by Taylor expansion, $$H(h_{a})\ll_{a,\deg h} H(h).$$
    
    Thus, $H(f_a)\ll_a H(f)\ll \prod_{\sigma \in {\rm Emb}(E,\QQ)} H(\sigma g)\leq B^{n_1}$. Since $a\in \ZZ$,  the coefficients $b_{i,a}$ of $g_a$ generate $E$. We also have that $f_a =\prod_{\sigma} \sigma g_a$, hence we may apply Lemma~\ref{lem_coefficient_height} and get 
    \begin{equation}\label{eq:bia}
        H(b_{i,a}) \ll_{a} H(f_a)^{1/n_1} \ll B.
    \end{equation}
    
    Applying Lemma~\ref{lem_gen} to the set $\{0, \ldots, C\}$ with $C\ll 1$, there is $0\leq a\leq C$ such that $b:=g(a) = g_a(0)=b_{0,a}$ generates $E$. By \eqref{eq:bia}, $H(b)\ll B$, as needed.
    
    Finally, if the $b_i$ are algebraic integers, then $f$ is integral. Hence $f_a$ is integral for any $a\in\ZZ$. Hence the coefficients of $g_a$ are algebraic integers.  
\end{proof}

For a number field $E$, let $\OO_E$ denote the set of algebraic integers in $E$, and let $\OO_{\bar{\QQ}}=\cup_E \OO_E$ denote the set of all algebraic integers. We define 
\begin{equation}\label{eq:bd_ht}
    S_E(n;Y) = \{b\in \OO_{\bar{\QQ}}:[E(b):E]=n,H(b)\leq Y\}.
\end{equation}

We shall need the following easy crude bound on the number of algebraic integers of a given degree and bounded Weil height. With extra work one can deduce from  \cite{chern2001distribution} an asymptotic formula, but it is not necessary for our purposes.

\begin{lemma}\label{lem_schmidt_quantitative_northcott}
    Let $Y \geq 1$. Then
    \[
        |S_{\QQ}(n;Y)|\ll Y^{n^2},
    \]
    as $Y\to \infty$.
\end{lemma}

\begin{proof}
    For each Galois orbit of $b_1,\ldots,b_n$ in $S_{\QQ}(n;Y)$, we attach the minimal polynomial $g=\prod_i (X-b_i)$. Then $g$ is monic with integers coefficients. By \eqref{naiveHvsWh}, $H(g) \ll \prod_i H(b_i)^{n}\leq Y^n$. Hence, the number of minimal polynomials is bounded by $C Y^{n^2}$, for some $C>0$, and so the number of $b$-s is bounded by $CnY^{n^2}\ll Y^{n^2}$, as claimed. 
\end{proof}

If we consider the algebraic integers of bounded height in a given extension we have the following sharper bound.

\begin{lemma}\label{lem_barroero_bounded_integers_in_given_field}
    Let $E$ be a number field of degree $n$ and denote by $\Delta_E$ its discriminant, and by $q_E=\rank(\OO_E^\times)$ the rank of the group of units of $\OO_E$. Let $Y \geq 1$. Then,
    \[
        |S_E(1;Y)|\ll \frac{1}{\sqrt{|\Delta_E|}} Y^n (\log Y)^{q_E} + Y^{n-1}(\log Y)^{q_E},
    \]
    as $Y\to \infty$. The implied constant depends only on $n$.  
\end{lemma}

\begin{proof}
    Let $r$, $2s$ be the number of real and complex embeddings of $E$, respectively. So $n=r+2s$ and $q_E=r+s-1\leq n$, by Dirichlet's unit theorem.  
    We apply  \cite[Theorem~1.1]{widmer2016integral} to the number field $E$ and with $n=e=1$ in \emph{loc.\ cit.} notation.  (Note that our $n$ is $m$ in \emph{loc.\ cit.}, and that $\OO_E(1,1)=\OO_E$, $N(\OO_E(1,1),Y)= S_E(1;Y)$,
    $t = q_E$, 
    $D_i = B_{E} \binom{t}{i}\frac{1}{i!}$, and 
    $B_E = \frac{2^r(2\pi)^s}{\sqrt{\Delta_E}}$.) Then,  Theorem~1.1 in 
    \emph{loc.\ cit.} gives
    \[
        |S_E(1;Y)|= B_{E} Y^n \sum_{i=0}^{t} \binom{t}{i} \frac{1}{i!} (\log Y^n)^i + O (Y^{n-1} (\log Y)^t),
    \]
    where the implied constant depends only on $n$.
    As $B_E \ll \frac{1}{\sqrt{\Delta_E}}$ and $\sum_{i=0}^t \binom{t}{i} \frac{1}{i!}\ll 1$, we get that,
    \[
        |S_E(1;Y) |\ll \frac{Y^{n}(\log Y)^t}{\sqrt{\Delta_E}} + Y^{n-1} (\log Y)^t,
    \]
    as needed.
\end{proof}

\subsection{Technical lemma}\label{sec:UpperBounds_2}
We will derive the theorems of the paper from the following lemma. 
\begin{lemma}\label{lem:mathcalF_bounds}
    Let  $k,n_1,n_2$ be positive integers, let $\Upsilon\ll L^{\frac{1}{n_1}}$, and put 
    \begin{align*}
        \mathcal{F}=\mathcal{F}_{k,n_1}(L) &= \bigg\{ E/\QQ : [E:\QQ]=n_1,\ E'\subsetneq E \Rightarrow [E':\QQ]\leq k, \\ &\qquad\qquad \mbox{ and}\quad  \exists b\in S_E(1;\Upsilon) \mbox{ such that }
        E=\QQ(b)\bigg\}.
    \end{align*}
    For $t\geq \frac{k}{n_1}$, we have 
    \[
        \sum_{E\in\mathcal{F}} N_E^{n_2} \ll L^{\nu_t+\epsilon},
    \]
    where $N_E = |S_E(1;\Upsilon)|$ and  $\nu_t = \max\{\frac12+\frac{1}{n_1} + n_2, n_1+tn_2, n_1+(1-\frac{1}{n_1})(n_2-1)\}$.
\end{lemma}

\begin{proof}
    We break the sum into three parts, and handle each by a different method. Let
    \begin{align*}
        & \mathcal{F}_1 := \{E\in \mathcal{F}:|\Delta_E|\leq L^{\frac{2}{n_1}}\}, \\
        & \mathcal{F}_2 := \{E\in \mathcal{F}:  N_E \leq L^{t+\epsilon}\}, \\
        & \mathcal{F}_3 := \{E\in \mathcal{F}: |\Delta_E|>  L^{\frac{2}{n_1}} \quad \mbox{and}\quad N_E > L^{t+\epsilon}\}.
    \end{align*}
    So, $\mathcal{F} = \mathcal{F}_1\cup \mathcal{F}_2\cup \mathcal{F}_3$; hence
    \[
        \sum_{E\in \mathcal{F}}N_E^{n_2} \leq  \sum_{E\in \mathcal{F}_1}N_E^{n_2}+\sum_{E\in \mathcal{F}_2}N_E^{n_2}+\sum_{E\in \mathcal{F}_3}N_E^{n_2} .
    \]
    To finish the proof it suffices to bound each term by $L^{\nu_t+\epsilon}$. 
    
    For the first term, we apply the union bound: By Schmidt\footnote{Today there are improvements on Schmidt, see \cite{bhargava2022improvement,lemke2022upper}.Applying these improvements does not improve our results, although it might improve some of the auxiliary bounds.} \cite[Equation (1.1)]{schmidt1995number}, $|\mathcal{F}_1|\ll L^{\frac{2}{n_1}\frac{n_1+2}{4}}=L^{\frac{1}{2}+\frac{1}{n_1}}$. By Lemma~\ref{lem_barroero_bounded_integers_in_given_field}, $N_E\ll \Upsilon^{n_1 +\epsilon'} \ll L^{1+\epsilon'}$. Hence,
    \[
        \sum_{E\in \mathcal{F}_1} N_E^{n_2} \ll L^{\frac{1}{2}+\frac{1}{n_1}+n_2+\epsilon} \leq  L^{\nu_t+\epsilon},
    \]
    where $\epsilon=\epsilon'n_1$.

    For the second term, we also use the union bound: By Lemma~\ref{lem_schmidt_quantitative_northcott},
    \[
        |\mathcal{F}_2|\leq |\mathcal{F}|\leq |S_{\QQ}(n_1;\Upsilon)|\ll|S_{\QQ}(n_1;L^{\frac{1}{n_1}})|\ll L^{n_1}.
    \]
    By the definition of $\mathcal{F}_2$, 
    \[
        \sum_{E\in \mathcal{F}_2} N_E^{n_2} \leq |\mathcal{F}_2|\cdot L^{tn_2+\epsilon}\ll L^{n_1+tn_2+\epsilon}\leq  L^{\nu_t+\epsilon}.
    \]
    
    Finally, we bound the third term. Let $0<\delta < \epsilon$. Let $E\in \mathcal{F}_3$. For a proper subfield $\QQ\subseteq E'\subsetneq E$, we have $[E':\QQ]\leq k$, so by   Lemma~\ref{lem_barroero_bounded_integers_in_given_field}, we get that  $|S_{E'}(1;\Upsilon)|\ll|S_{E'}(1;L^{\frac{1}{n_1}})|\ll L^{\frac{k}{n_1}+\delta}$. 
    We apply Lemma~\ref{lem_covers_bound} to the cover $T =\{S_E: E\in\mathcal{F}_3\}$ of 
    $S:=\cup_{E\in\mathcal{F}_3}S_E$, where  $S_E:=S_E(1;\Upsilon)$. 
    Note that  $X:=|S| \ll |S_{\QQ}(n_1;\Upsilon)|\ll L^{n_1}$ by Lemma~\ref{lem_schmidt_quantitative_northcott}.   Moreover, $Z\leq |S_E|\leq Y$, where by  Lemma~\ref{lem_barroero_bounded_integers_in_given_field}, $Y\ll  L^{1-\frac{1}{n_1}+\epsilon}$  and, by the definition of $\mathcal{F}_3$,    $Z=L^{t+\epsilon}$ . Finally, we have
    \[
        S_E\cap(\cup_{E\neq F\in \mathcal{F}_3}S_F)\subset \cup_{E'\subsetneq E}S_{E'}.
    \]
    Hence,  $|S_E\cap(\cup_{E\neq F\in \mathcal{F}_3}S_F)|\ll R=L^{t+\delta}=o(Z)$. Thus, 
    \[
        M_{n_2}(T)=\sum_{E\in \mathcal{F}_3} N_E^{n_2} \ll XY^{n_2-1} = L^{n_1+(1-\frac{1}{n_1})(n_2-1)+\epsilon} \leq L^{\nu_t+\epsilon},
    \]
    as required.
\end{proof}

\subsection{Proof of upper bounds}
Let $n=n_1 n_2$, $n_1,n_2>1$. If $n_2=2$ and $n\geq 10$ put $m = \frac{n}{2}-\frac{2}{n}+1$. Otherwise, put
\[
    m = \max\bigg\{n_2+\frac{1}{2}+\frac{1}{n_1},n_1+(1-\frac{1}{n_1})(n_2-1),n_1+\frac{n_2}{2}\bigg\}.
\]
Denote by $\Imp(n_1,n_2;L)$ the set of monic irreducible polynomials with integer coefficients of height at most $L$ with transitive $(n_1,n_2)$-imprimitive Galois group. To prove Theorem~\ref{thm:imprimitive}, it suffices to prove that  $|\Imp(n_1,n_2;L)| \ll L^{m+\epsilon}$. 

Recall that a transitive group $G$ is $(n_1,n_2)$-imprimitive if and only if the stabilizer contains a subgroup of index $n_1$ \cite[Theorem 7.5]{wielandt2014finite}. By the Galois correspondence, this implies that if $f$ is an irreducible polynomial of degree $n=n_1n_2$, and $\alpha$ is its root, then $G_f$ is $(n_1,n_2)$-imprimitive if and only if  there exists a subfield $\QQ\subset E\subset \QQ(\alpha)$ of degree $n_1$ over $\QQ$. 

In this case, let $g\in E[x]$ be the minimal polynomial of $\alpha$ over $E$.   Then the set $\{\sigma g : \sigma\in {\rm Emb}(E, \mathbb{C})\}$ consists of pairwise coprime polynomials, and
\begin{equation}\label{imprimitive_product}
    f = \prod_{\sigma\in {\rm Emb}(E, \mathbb{C})} \sigma g.  
\end{equation}
(Here we used that the RHS is a monic polynomial with rational coefficients of degree equal to $\deg f$ and with $\alpha$ as a zero.)

\begin{lemma}\label{lem_pairs}
    There exists a constant $C>0$, depending only on $n$, such that for every $f\in \Imp(n_1,n_2;L)$ there exists a pair $(E,g)$, where $E$ is a number field of degree $n_1$ generated by some $b\in S_E(1;CL^{\frac{1}{n_1}})$, and $g= X^{n_2}+\sum_{i=0}^{n_2-1} b_i X^i\in E[X]$ divides $f$, and satisfies $H(b_i)\ll CL^{\frac{1}{n_1}}$. Moreover, the number of such pairs $(E,g)$ is $\ll 1$.
\end{lemma}

\begin{proof}
    Since $f$ has $\ll 1$ factors, and the splitting field of $f$ has $\ll 1$ subfields, there are $\ll 1$ such pairs $(E,g)$. 
    
    We take $E$ and $g$ as in \eqref{imprimitive_product}. Since $\{\sigma g : \sigma\in {\rm Emb}(E, \mathbb{C})\}$ are coprime, the orbit of $(b_0,\ldots, b_{n_2-1})$ is of length $n_1$, hence this tuple generates $E/\QQ$. So we may apply Lemma~\ref{lem_coefficient_height} to get that $H(b_i)\ll L^{\frac{1}{n_1}}$, and Lemma~\ref{lem_generatorsmallheight} to obtain a generator $b\in\OO_E$ of $E/\QQ$ such that $H(b)\ll \max(H(b_i))\ll L^{\frac{1}{n_1}}$.
\end{proof}

\subsubsection{Proof of the upper bound in Theorem~\ref{thm:C_2WrS_n}}
The action of $G:=(C_2\wr S_{\frac{n}{2}})\cap A_n$ on the $\frac{n}{2}$ blocks of size $2$ induces a surjective map $G\to S_{\frac{n}{2}}$; in particular, it is primitive. So, the proof follows from the following  more general statement (note that $\nu(\frac{n}{2},2) = \frac{n}{2}+1-\frac2n$ for $n\geq 6$):

\begin{proposition}\label{prop:Primitive}
    Let $G$ be a transitive $(n_1,n_2)$-imprimitive group such that the induced action on the $n_1$ blocks of imprimitivity is primitive. Then 
    \[
        \Prob(G_f=G) \ll L^{-n+\nu+\epsilon},
    \]  
    with $\nu = \nu(n_1,n_2) = \max\{\frac12+\frac{1}{n_1} + n_2, n_1+\frac{n_2}{n_1}, n_1+(1-\frac{1}{n_1})(n_2-1)\}$.
\end{proposition}

\begin{proof}
    We need to bound the number $N$ of monic $f\in\ZZ[X]$ of degree $n$ having Galois group $G_f=G$ and height $H(f)\leq L$. For such an $f$, let $\alpha$ be a root of $f$, let $\alpha\in \Lambda$ be a block of size $n_2$, and let $E$ be the field fixed by the stabilizer of $\Lambda$. Then $\QQ\subseteq E\subseteq \QQ(\alpha)$, $[\QQ(\alpha):\QQ]=n$ and  $[E:\QQ]=n_1$. Moreoever, the assumption that $G$ acts primitively on the set $\{g\Lambda : g\in G\}$ implies that $E/\QQ$ is minimal. 
    
    Let $C>0$ be the constant of Lemma~\ref{lem_pairs}, and $\Upsilon = CL^{\frac{1}{n_1}}$. We want to apply Lemma~\ref{lem:mathcalF_bounds} with $k=1$ and $t=\frac{1}{n_1}$. For this, we note that 
    by Lemma~\ref{lem_coefficient_height} and the above paragraph, $N$ is bounded by the  number of pairs $(E,g)$, where $E\in \mathcal{F}_{1,n_1}(L)$ and $g\in E[X]$ is a monic polynomial with coefficients in $S_E(1;\Upsilon)$.
    Since there are $N_E:=|S_E(1;\Upsilon)|$ options for each coefficient of $g$, we get the bound 
    \[
        N\leq \sum_{E\in \mathcal{F}_{1,n_1}(L)} N_E^{n_2} \ll L^{\nu+\epsilon},
    \]
    as needed. 
\end{proof}

\subsubsection{Proof of Theorem \ref{thm:imprimitive}}
Let $C>0$ be the constant of Lemma~\ref{lem_pairs}, $\Upsilon = CL^{\frac{1}{n_1}}$, and $k=\lfloor\frac{n_1}{2}\rfloor$. By Lemma~\ref{lem_coefficient_height}, $|\Imp(n_1,n_2;L)|$ is bounded by the  number of pairs $(E,g)$, where $E\in \mathcal{F}_{k,n_1}(L)$, and $g\in E[X]$ is a monic polynomial with coefficients in $S_E(1;\Upsilon)$. Since there are $N_E:=|S_E(\Upsilon)|$ options for each coefficient of $g$, we get the bound 
\[
    |\Imp(n_1,n_2;L)| \leq \sum_{E\in \mathcal{F}_{k,n_1}(L)} N_E^{n_2}.
\]
Thus by Lemma~\ref{lem:mathcalF_bounds} applied with $k=\lfloor \frac{n_1}{2}\rfloor$ and $t=\frac{1}{2}$, we get that 
\[
    \Prob(G_f \mbox{ is transitive $(n_1,n_2)$-imprimitive}) \ll L^{-n+\nu +\epsilon}
\]
with $\nu = \max\{\frac12+\frac{1}{n_1} + n_2, n_1+\frac{n_2}{2}, n_1+(1-\frac{1}{n_1})(n_2-1)\}$. This is the asserted bound if $n_2>2$ or $n<10$. 

To this end, assume that $n_2=2$ and $n\geq 10$. We condition on the largest $k<n_1$ such that $E$ contains a subfield $E'$ with $[E':\QQ]=k$. 
If $k=1$, that is, the action of $G_f$ on the $n_1$-blocks is primitive, then by Proposition~\ref{prop:Primitive},  
\[
    \Prob(G_f \mbox{ is transitive $(n_1,n_2)$-imprimitive} \mid k=1) \ll L^{-\frac n2+1-\frac2n+\epsilon}.
\]
Next, we assume that $k=2$. By Lemma~\ref{lem:mathcalF_bounds} and the bound of the probability in terms of sums of $N_E^{n_2}$, 
\begin{align*}
    \Prob(G_f \mbox{ is transitive $(n_1,n_2)$-imprimitive} \mid k=2) &\ll \frac{1}{L^n}\sum_{E\in \mathcal{F}_{k,n/2}} N_E^{n_2} \\ &\ll L^{-\frac n2+1-\frac2n+\epsilon}.
\end{align*}
(Here we used that $n\geq 10$.)
Finally, if $k>2$, then $G_f$ is $(k,\frac{n}{k})$-imprimitive, and we apply the first paragraph, and note that the corresponding $\nu$ satisfies
\[
\nu = \max\{\frac12+\frac{1}{k} + \frac nk, k+\frac{n}{2k}, k+(1-\frac{1}{k})(\frac{n}{k}-1)\}\leq \frac n2+1-\frac2n.
\]

\qed

\bibliographystyle{plain}


\begin{thebibliography}{10}

\bibitem{anderson2021quantitative}
Theresa~C Anderson, Ayla Gafni, Robert J~Lemke Oliver, David Lowry-Duda, George
  Shakan, and Ruixiang Zhang.
\newblock Quantitative {H}ilbert irreducibility and almost prime values of
  polynomial discriminants.
\newblock {\em International Mathematics Research Notices}, 2023(3): 2188--2214, 2023.

\bibitem{bary2020chebotarev}
Lior Bary-Soroker, Ofir Gorodetsky, Taelin Karidi, and Will Sawin.
\newblock Chebotarev density theorem in short intervals for extensions of
  $\mathbb{F}_q(t)$.
\newblock {\em Transactions of the American Mathematical Society},
  373(1):597--628, 2020.

\bibitem{bary2020irreducibility}
Lior Bary-Soroker, Dimitris Koukoulopoulos, and Gady Kozma.
\newblock Irreducibility of random polynomials: general measures.
\newblock {\em Inventiones mathematicae}, 223:1041--1129, 2023.

\bibitem{bary2020irreducible}
Lior Bary-Soroker and Gady Kozma.
\newblock Irreducible polynomials of bounded height.
\newblock {\em Duke Mathematical Journal}, 169(4):579--598, 2020.

\bibitem{bhargava2021galois}
Manjul Bhargava.
\newblock Galois groups of random integer polynomials and van der {W}aerden's
  {C}onjecture.
\newblock {\em arXiv:2111.06507}, 2021.

\bibitem{bhargava2022improvement}
Manjul Bhargava, Arul Shankar, and Xiaoheng Wang.
\newblock An improvement on {S}chmidt’s bound on the number of number fields of
  bounded discriminant and small degree.
\newblock {\em Forum of Mathematics, Sigma}, 10:13pp, Paper no.e86, 2022.

\bibitem{bombieri2007heights}
Enrico Bombieri and Walter Gubler.
\newblock {\em Heights in {D}iophantine geometry}.
\newblock New Mathematical Monograph, 4, Cambridge University Press, 2006.

\bibitem{breuillard2019irreducibility}
Emmanuel Breuillard and P{\'e}ter~P Varj{\'u}.
\newblock Irreducibility of random polynomials of large degree.
\newblock {\em Acta Mathematica}, 223(2):195--249, 2019.

\bibitem{chela1963reducible}
R~Chela.
\newblock Reducible polynomials.
\newblock {\em Journal of the London Mathematical Society}, 1(1):183--188,
  1963.

\bibitem{chern2001distribution}
Shey-Jey Chern and Jeffrey~D Vaaler.
\newblock The distribution of values of {M}ahler's measure.
\newblock {\em Journal f{\"u}r die reine und angewandte Mathematik},
  2001(540):1--47, 2001.

\bibitem{Chow2020EnumerativeQuartics}
Sam Chow and Rainer Dietmann.
\newblock {E}numerative {G}alois theory for cubics and quartics.
\newblock {\em Advances in Mathematics}, 372:37pp, 2020.

\bibitem{chow2021towards}
Sam Chow and Rainer Dietmann.
\newblock Towards van der {W}aerden's conjecture.
\newblock {\em Transactions of the American Mathematical Society}, 376(4):2739--2785, 2023.

\bibitem{Cohen}
Stephen~D. Cohen.
\newblock The distribution of galois groups and {H}ilbert’s irreducibility
  theorem.
\newblock {\em Proceeding of the 
London Mathematical Society}, 43(2):227--250, 1981.

\bibitem{dietmann2013probabilistic}
Rainer Dietmann.
\newblock Probabilistic {G}alois theory.
\newblock {\em Bulletin of the London Mathematical Society}, 45(3):453--462,
  2013.

\bibitem{eberhard2022characteristic}
Sean Eberhard.
\newblock The characteristic polynomial of a random matrix.
\newblock {\em Combinatorica}, pages 1--37, 2022.

\bibitem{ferber2022random}
Asaf Ferber, Vishesh Jain, Ashwin Sah, and Mehtaab Sawhney.
\newblock Random symmetric matrices: rank distribution and irreducibility of
  the characteristic polynomial.
\newblock {\em Mathematical Proceedings of the Cambridge Philosophical
  Society}, 174(2):233-246.

\bibitem{gallagher1973large}
Patrick~X Gallagher.
\newblock The large sieve and probabilistic {G}alois theory.
\newblock {\em Proceedings of Symposia in Pure Mathematics}, XXIV:91--101, 1973.

\bibitem{Hilbert}
David Hilbert.
\newblock Ueber die {I}rreducibilit¨at ganzer rationaler functionen mit
  ganzzahligen {C}oefficienten.
\newblock {\em Journal f{\"u}r die reine und angewandte Mathematik}, 110:104--129, 1892.

\bibitem{humphreys1990reflection}
James~E Humphreys.
\newblock {\em Reflection groups and {C}oxeter groups}.
\newblock Number~29. Cambridge University Press, 1990.

\bibitem{Knobloch}
Hans-Wilhelm Knobloch.
\newblock Die {S}eltenheit der reduziblen {P}olynome.
\newblock {\em Jahresbericht der Deutschen Mathematiker}, 59(1):12--19, 1956.

\bibitem{kuba2009distribution}
Gerald Kuba.
\newblock On the distribution of reducible polynomials.
\newblock {\em Mathematica Slovaca}, 59(3):349--356, 2009.

\bibitem{LLT}
Aaron Landesman, Robert~J. Lemke-Oliver, and Frank Thorne.
\newblock Improved lower bounds for the number of fields with alternating
  {G}alois group.
\newblock {\em Bulletin of the London Mathematical Society}, 53(4):1159--1173,
  2021.

\bibitem{lemke2022upper}
Robert~J Lemke~Oliver and Frank Thorne.
\newblock Upper bounds on number fields of given degree and bounded
  discriminant.
\newblock {\em Duke Mathematical Journal}, 171(15):3077--3087, 2022.

\bibitem{mahler1962some}
Kurt Mahler.
\newblock On some inequalities for polynomials in several variables.
\newblock {\em Journal of the London Mathematical Society}, 37(1):341--344, 1962.

\bibitem{mahler1963two}
Kurt Mahler et~al.
\newblock On two extremum properties of polynomials.
\newblock {\em Illinois Journal of Mathematics}, 7(4):681--701, 1963.

\bibitem{schmidt1995number}
Wolfgang~M Schmidt.
\newblock Number fields of given degree and bounded discriminant.
\newblock {\em Ast{\'e}risque}, 228(4):189--195, 1995.

\bibitem{van1936seltenheit}
Bartel~L van~der Waerden.
\newblock Die {S}eltenheit der reduziblen {G}leichungen und der {G}leichungen
  mit {A}ffekt.
\newblock {\em Monatshefte f{\"u}r Mathematik und Physik}, 43(1):133--147,
  1936.

\bibitem{widmer2011number}
Martin Widmer.
\newblock On number fields with nontrivial subfields.
\newblock {\em International Journal of Number Theory}, 7(03):695--720, 2011.

\bibitem{widmer2016integral}
Martin Widmer.
\newblock Integral points of fixed degree and bounded height.
\newblock {\em International Mathematics Research Notices},
  2016(13):3906--3943, 2016.

\bibitem{wielandt2014finite}
Helmut Wielandt.
\newblock {\em Finite permutation groups}.
\newblock Academic press, 2014.

\bibitem{xiao2022monic}
Stanley~Yao Xiao.
\newblock On monic abelian cubics.
\newblock {\em Compositio Mathematica}, 158(3):550--567, 2022.

\bibitem{zywina2010hilbert}
David Zywina.
\newblock Hilbert's irreducibility theorem and the larger sieve.
\newblock {\em arXiv:1011.6465}, 2010.

\end{thebibliography}

\end{document}